\documentclass[reqno,11pt]{amsart}
\usepackage[foot]{amsaddr}
\usepackage{bbold}
\usepackage{charter}
\usepackage{enumitem}
\usepackage[margin=1in]{geometry}
\usepackage[colorlinks=true,linktoc=all,linkcolor=blue,citecolor=blue]{hyperref}

\theoremstyle{plain}
\newtheorem{theorem}{Theorem}[section]
\newtheorem*{theorem*}{Theorem}
\newtheorem{proposition}{Proposition}[theorem]
\newtheorem{lemma}[theorem]{Lemma}
\newtheorem{corollary}[theorem]{Corollary}
\theoremstyle{definition}
\newtheorem{case}{Case}
\newtheorem{definition}[theorem]{Definition}
\newtheorem*{definition*}{Definition}
\newtheorem{remark}[theorem]{Remark}
\newtheorem{example}[theorem]{Example}
\newtheorem{conjecture}[theorem]{Conjecture}
\numberwithin{equation}{section}
\everymath{\displaystyle}
\allowdisplaybreaks

\setlist[enumerate,1]{label = \upshape(\roman*), ref = (\arabic*)}

\newcommand{\C}{\mathbb{C}}
\newcommand{\D}{\mathbb{D}}
\newcommand{\N}{\mathbb{N}}
\newcommand{\R}{\mathbb{R}}
\newcommand{\T}{\mathbb{T}}
\newcommand{\X}{\mathcal{X}}

\newcommand{\Y}{\mathcal{Y}}

\newcommand{\Hardyp}{\T_p}
\newcommand{\HardypT}{\Hardyp(T)}
\newcommand{\Lip}{\mathcal{L}}
\newcommand{\LipT}{\mathcal{L}(T)}
\newcommand{\Lmuinf}{L_{\hspace{-.2ex}\mu}^{\infty}}
\newcommand{\LmuinfT}{L_{\hspace{-.2ex}\mu}^{\infty}(T)}
\newcommand{\bigchi}{\mbox{\large$\chi$}}
\newcommand{\ch}{\mathrm{ch}}
\newcommand{\metric}{\mathrm{d}}

\title{The differentiation operator on discrete function spaces of a tree}

\author{Robert F. Allen\textsuperscript{1} and Colin M. Jackson\textsuperscript{2}}
\address{\textsuperscript{1}Department of Mathematics and Statistics, University of Wisconsin-La Crosse}
\address{\textsuperscript{2}Department of Mathematics, University of Colorado Boulder}
\email{rallen@uwlax.edu, colin.jackson@colorado.edu}


\begin{document}

\begin{abstract} 
In this paper, we study the differentiation operator acting on discrete function spaces; that is spaces of functions defined on an infinite rooted tree.  We discuss, through its connection with composition operators, the boundedness and compactness of this operator.  In addition, we discuss the operator norm and spectrum, and consider when such an operator can be an isometry.  We then apply these results to the operator acting on the discrete Lipschitz space and weighted Banach spaces, as well as the Hardy spaces defined on homogeneous trees.
\end{abstract}

\maketitle

\section{Introduction}\label{Section:Introduction}
Much work has been done in defining function spaces on discrete structures, such as infinite trees.  While this is not new, many different approaches have been taken in the past.  Beginning with the work of Colonna and Easley, the study of the Lipschitz space on an infinite rooted tree began a line of research into which this paper fits nicely.  Among many of these spaces, the derivative plays a defining role.  The Lipschitz space $\Lip$ defined in \cite{ColonnaEasley:2010} consists of functions on a tree with bounded derivative.  Through the study of the multiplication operators on $\Lip$, a family of spaces was defined called the iterated logarithmic Lipschitz spaces $\Lip^{(k)}$, again defined in terms of a weighted derivative being bounded (see \cite{AllenColonnaEasley:2012}).  In addition, the Zygmund space on a tree was defined in \cite{Locke:2016} as the space of functions whose derivative is contained in $\Lip$.  We see that differentiation is key in the study of many of these discrete function spaces.  However, differentiation itself has not been studied on these spaces of infinite trees.

The purpose of this paper is to study differentiation as an operator on discrete functions spaces of an infinite tree.  Unlike differentiation acting on many classical function spaces of $\D = \{z \in \C : |z| < 1\}$, this operator acting on discrete spaces is often bounded.  Since it is prevalent in the definition of many discrete spaces, it is a good idea to understand more properties of this operator.  

Also, we hope this paper will inspire further study of operators on discrete function spaces.  With more knowledge of the derivative, operators comprising products of differentiation, multiplication, and composition operators can be further studied.  These types of operators are currently being studied when acting on or between many of the classical spaces (see for example \cite{ColonnaSharma:2013,FatehiHammond:2020,FatehiHammond:2021,Sharma:2011,SharmaRajSharma:2012}).  In addition, we also hope that new spaces can be identified utilizing the derivative much in the way that the weighted Lipschitz space and the Zygmund space have been.  There are many examples of these types of spaces in the classical setting, such as the $S^p$ spaces, functions whose derivative is contained in $H^p$.

\subsection{Organization of the paper} In Section \ref{Section:Operator}, we study the differentiation operator $D$ acting on an arbitrary discrete function space.  We provide necessary and sufficient conditions that determine the boundedness of $D$.  While establishing a connection between the differentiation and composition operators, we establish estimates on the norm of $D$ and the spectrum.  We determine conditions on the function space for which $D$ is not an isometry as well as completely determine the eigenvalues.  Finally, we provide sufficient conditions on the function space for which $D$ is not compact.  

In Sections \ref{Section:LipschitzSpace},\ref{Section:WeightedBanachSpace}, and 
\ref{Section:HardySpaces}, we apply the results from Section \ref{Section:Operator} to the Lipschitz space, the weighted Banach spaces, and the discrete Hardy spaces, respectively.  When advantageous, we utilize known results of composition operators.  In all three sections we determine the operator norm of $D$, which lends concrete evidence to support Conjecture \ref{Conjecture:OperatorNorm}. We also show $D$ not to be an isometry on any of these spaces.  Finally, we discuss on what spaces $D$ is not compact.  For the differentiation operator acting on the Lipschitz space, we completely determine the spectrum.  

We end the paper with Section \ref{Section:OpenQuestions}, providing open questions posed throughout the manuscript.  We hope these questions inspire readers to further the study of operators on discrete functions spaces.  The study of operators on discrete spaces does not depend on advanced topics, such as measure theory and complex analysis, that are required in the study on classical spaces.

\section{Differentiation on Discrete Function Spaces}\label{Section:Operator}

By a \textit{tree} $T$ we mean a locally finite, connected, and simply-connected graph, which, as a set, we identify with the collection of its vertices.  A tree $T$ is \textit{homogeneous} if every vertex has the same number of neighbors.  For $q$ in $\N$, a $(q+1)$-homogeneous tree is one where every vertex has $q+1$ neighbors.  

Two vertices $v$ and $w$ are called \textit{neighbors} if there is an edge $[v,w]$ connecting them, and we use the notation $v\sim w$. A vertex is called \textit{terminal} if it has a unique neighbor. A \textit{path} is a sequence of vertices $[v_0,v_1,\dots]$ such that $v_k\sim v_{k+1}$ and  $v_{k-1}\ne v_{k+1}$, for all $k$.  Define the \textit{length} of a finite path $[v=v_0,v_1,\dots,w=v_n]$ to be the number  of edges $n$ connecting $v$ to $w$. The \textit{distance} between vertices $v$ and $w$ is the length $\metric(v,w)$ of the unique path connecting $v$ to $w$.

Given a tree $T$ rooted at $o$, the \textit{length} of a vertex $v$ is defined as $|v|=\metric(o,v)$. For a vertex $v\in T$, a vertex $w$ is called a \textit{descendant} of $v$  if $v$ lies in the path from $o$ to $w$. The vertex $v$ is then called an \textit{ancestor} of $w$.  For $v \in T$ with $v \neq o$, we denote by $v^-$ the unique neighbor which is an ancestor of $v$.  The vertex $v$ is called a \textit{child} of $v^-$. 
For $v\in T$, the set $\ch(v)$ consists of all children of $v$, and the set $S_v$ consists of $v$ and all its descendants, called the \textit{sector} determined by $v$. The set $T\setminus\{o\}$ will be denoted by $T^*$. We denote the open ball centered at vertex $v$ of radius $n \in \N$ to be the set $B(v,n) = \{w \in T : \metric(w,v) < n\}$, and the closed ball centered at vertex $v$ of radius $n \in \N$ to be the set $\overline{B(v,n)} = \{w \in T : \metric(w,v) \leq n\}.$ By a \textit{function on a tree} we mean a complex-valued function on the set of its vertices.

In this paper, we shall assume the tree $T$ to be without terminal vertices (and hence infinite), and rooted at a vertex $o$.  We define the \textit{derivative} of a function $f$ on $T$ as 
\[f'(v) = \begin{cases}f(v)-f(v^-) & \text{if $v \neq o$},\\\hfil 0 & \text{if $v = o$.}\end{cases}\]  By defining the \textit{backward shift} map $b:T\to T$ as  \[b(v) = \begin{cases}v^- & \text{if $v \neq o$}\\ o & \text{if $v = o$,}\end{cases}\] the derivative of a function on $T$ can be written as \[f'(v) = f(v)-f(b(v))\] for all $v \in T$. The derivative of a function on $T$ has many of the expected properties of the derivative from calculus.

\begin{lemma}\label{Lemma:DerivativeProperties} 
For each $\alpha$ in $\C$ and functions $f$ and $g$ on tree $T$,
\begin{enumerate} 
\item $(f+g)' = f'+g'$ and $(\alpha f)' = \alpha f'$.
\item $f' \equiv 0$ if and only if $f$ is constant.
\item if $f' \equiv g'$ then there exists a positive constant $C$ such that $f(v) = g(v) + C$ for all $v$ in $T$.
\end{enumerate}
\end{lemma}

\begin{proof}
First, let $v$ be an arbitrary vertex in $T^*$. Observe
\[\begin{aligned}
(f+g)'(o) &= 0 = f'(o) + g'(o)\\
(f+g)'(v) &= (f+g)(v) - (f+g)(b(v)) = f(v) + g(v) - f(b(v)) - g(b(v)) = f'(v) + g'(v)
\end{aligned}\]
and
\[\begin{aligned}
(\alpha f)'(o) &= 0 = \alpha f'(o)\\
(\alpha f)'(v) &= (\alpha f)(v) - (\alpha f)(b(v)) = \alpha f(v) - \alpha f(b(v)) = \alpha f'(v).
\end{aligned}\]

Next, suppose $f:T \to \C$ is constant, that is $f(v) = f(o)$ for all $v$ in $T$.  By definition, $f'(o) = 0$.  Let $w$ be an arbitrary vertex in $T^*$.  Then $f'(w) = f(w) - f(b(w)) = 0$.  So $f'$ is identically 0.  Now, suppose $g:T \to \C$ is such that $g'$ is identically 0 on $T$.  Let $w$ in $T^*$ with $|w| = 1$.  Then $0 = g'(w) = g(w) - g(o)$.  So $g(w) = g(o)$ for all $|w| = 1$.  By induction, we have $g(v) = g(o)$ for all $v$ in $T^*$.  Thus $g$ is constant. Finally, (iii) follows immediately from (i) and (ii), as $f'(v) = g'(v)$ for all $v$ in $T$ implies $f-g$ is constant.
\end{proof}

In order to define differentiation as an operator, we make use of the definition of a functional Banach space.
\begin{definition}[{\cite[Definition 1.1]{CowenMacCluer:1995}}]
A Banach space of complex valued functions on a set $\Omega$ is called a \textit{functional Banach space on $\Omega$} if the vector operations are the pointwise operations, $f(x) = g(x)$ for each $x$ in $\Omega$ implies $f=g$, $f(x)=f(y)$ for each function in the space implies $x=y$, and for each $x \in \Omega$, the linear functional $f \mapsto f(x)$ is continuous.
\end{definition}

\noindent A \textit{discrete function space on a tree $T$}, denoted by $\X(T)$ or simply $\X$, is a functional Banach space, whose elements are functions on tree $T$, endowed with norm $\|\cdot\|_{\X}$.  If $\X$ and $\Y$ are discrete functions spaces, we define the \textit{differentiation operator} $D:\X \to \Y$ as \[Df = f'\] for all $f$ in $\X$. The differentiation operator is linear by Theorem \ref{Lemma:DerivativeProperties}.  

While this section is concerned with the differentiation operator acting on an arbitrary discrete function space $\X$, we will apply these results to three specific spaces in Sections \ref{Section:LipschitzSpace}, \ref{Section:WeightedBanachSpace}, and 
\ref{Section:HardySpaces}.  We wish to define these spaces now to bring the results of this section into context.  

\begin{definition}\label{Definition:SpecificSpaces} Let $T$ be a tree.  
\begin{enumerate}
\item The \textit{Lipschitz space} $\LipT$ is the set
\[\left\{f:T \to \C\;\left|\; \sup_{v \in T^*}\;|f'(v)| < \infty\right.\right\}.\]
\item For a positive function $\mu:T \to \R$, called a \textit{weight}, the \textit{weighted Banach space} $\LmuinfT$ is the set \[\left\{f:T \to \C\;\left|\; \sup_{v \in T}\;\mu(v)|f(v)| < \infty\right.\right\}.\]
\item For $q$ in $\N$, let $T$ be a $(q+1)$-homogeneous tree; that is a tree with every vertex having $q+1$ neighbors. The \textit{Hardy space} $\HardypT$, for $1 \leq p < \infty$ is the set 
\[\left\{f:T\to\C\;\left|\;\sup_{n \in \N_0} M_p(n,f)<\infty\;\right.\right\},\]
where $M_p(0,f) = |f(o)|$ and for $n$ in $\N$
\[M_p(n,f) = \left(\frac{1}{(q+1)q^{n-1}}\sum_{|v|=n} |f(v)|^p\right)^{1/p}.\]
\end{enumerate}
\end{definition}
\noindent Discussions as to these spaces being functional Banach spaces are provided in their respective sections.

When studying the differentiation operator on a discrete function space, the following representation of $D$ will be very useful.  

\begin{lemma}\label{Lemma:DRepresentation} 
If $C_b$ maps a discrete function space $\X$ into itself, then $D$ maps $\X$ into $\X$ and \[D = I-C_b,\] where $I$ is the identity operator on $\X$ and $C_b$ is the composition operator induced by the backward shift map.
\end{lemma}

\begin{proof}
Suppose $f$ is a function in $\X$ and $v$ in $T^*$.  Observe 
\[(Df)(o) = f'(o) = 0 = f(o) - f(b(o)) = (If)(o) - (C_bf)(o) = ((I-C_b)f)(o)\]
and
\[(Df)(v) = f'(v) = f(v) - f(b(v)) = (If)(v) - (C_bf)(v) = ((I-C_b)f)(v).\] So $Df = (I-C_b)f$, which shows $D$ maps $\X$ into $\X$ and $D = I-C_b$.
\end{proof}

\subsection{Boundedness and Operator Norm}
The authors are interested in the study of bounded linear operators. On classical spaces, it is often the case that differentiation is unbounded.  So, our first goal in this section is to develop criteria for which differentiation on discrete functions spaces can be bounded.  For the bounded differentiation operator, we then consider expressions for the operator norm, and when such an operator can be an isometry.

In the literature, especially seen for composition operators, the proof of boundedness involves showing an operator $A:X \to Y$ maps into $X$.  This type of argument relies on the Closed Graph Theorem (which we provide for completeness), and usually is made as a remark without proof. 

\begin{theorem}[{\cite[Closed Graph Theorem]{MacCluer:2009}}] 
If $X$ and $Y$ are Banach spaces and $A:X \to Y$ is linear, then $A$ is bounded if and only if $\mathrm{graph}(A) = \{(x,Ax): x \in X\}$ is closed in $X\times Y$.
\end{theorem}

\noindent We offer the following characterization of boundedness for the differentiation operator, as well as provide a proof.

\begin{theorem}\label{Theorem:BoundednessCGT} 
The operator $D$ is bounded on a discrete function space $\X$ if and only if $D$ maps $\X$ into $\X$.
\end{theorem}

\begin{proof}
As the forward direction follows by definition, it suffices to that that $D$ is bounded on $\X$ if $D$ maps $\X$ into $\X$, for which we will use the Closed Graph Theorem.  Let $\{f_n\}$ be a sequence in $\X$ converging to a function $f$ in $\X$; additionally, suppose the sequence $\{Df_n\}$ converges in $\X$ to a function $g$. We then need to show $Df=g$.  

First note that $(Df_n)(o) = f'_n(o) = 0$ for all $n$ in $\N$.  Thus $g(o) = 0$, since norm convergence implies point-wise convergence in a functional Banach space.  Additionally, for an arbitrary vertex $v$ in $T^*$, observe \[(Df_n)(v) = f'_n(v) = f_n(v)-f_n(b(v)) = f_n(v) - K_{b(v)}(f_n)\] for the evaluation functional $K_{b(v)}$. Since $\X$ is a functional Banach space, $K_{b(v)}$ is a continuous linear functional.  Together with the hypothesis that $\{f_n\}$ converges to $f$, this implies that $\{K_{b(v)}(f_n)\}$ converges to $K_{b(v)}(f)$. Thus, for each $v$ in $T$, we have that $\{(Df_n)(v)\}$ converges to \[f(v) - K_{b(v)}(f) = f(v) - f(b(v)) = f'(v) = (Df)(v).\]  Again, since norm convergence implies pointwise convergence, we have that $\{(Df_n)(v)\}$ converges to $Df(v)$.  This implies that $(Df)(v)=g(v)$ for every $v$ in $T$ and thus $Df=g$ as desired.  Therefore, we conclude $D$ is bounded on $\X$ by the Closed Graph Theorem.
\end{proof}

\noindent We can also characterize the boundedness of the differentiation operator through Lemma \ref{Lemma:DRepresentation}.  

\begin{theorem}\label{Theorem:BoundednessCphi} 
The operator $D$ is bounded on a discrete function space $\X$ if and only if $C_b$ is bounded on $\X$.
\end{theorem}

\begin{proof}
Suppose $D$ is bounded on $\X$.  Then $C_b = I-D$ maps $\X$ into $\X$ and is bounded since $I$ is bounded on $\X$.  Likewise, if $C_b$ is bounded on $\X$, then $D$ maps $\X$ into $\X$.  By Theorem \ref{Theorem:BoundednessCGT}, $D$ is bounded on $\X$.
\end{proof}

\begin{remark}
Composition operators have already been studied on the Lipschitz space \cite{AllenColonnaEasley:2014}, the weighted Banach spaces \cite{AllenPons:2018}, and the Hardy spaces \cite{MuthukumarPonnusamy:2020}. Thus, characterizations of boundedness for the differentiation operator on these specific spaces utilize Theorem \ref{Theorem:BoundednessCphi} (see Theorems \ref{Theorem:BoundednessLip}, \ref{Theorem:BoundednessLmuinf}, and \ref{Theorem:BoundednessHardyp}). However, when one pays attention to the study of operators on these spaces, it is interesting to note that composition operators are not among the first to be studied.  In fact, it is typical for multiplication operators to first be studied on discrete spaces of the kind following the development of the Lipschitz space.  As new discrete spaces are defined, in the absence of results on composition operators,  Theorem \ref{Theorem:BoundednessCGT} can provide the means of characterizing boundedness of $D$. Theorem \ref{Theorem:BoundednessCphi} may then provide insight on conditions to characterize the boundedness of the composition operator $C_\varphi$ where $\varphi$ is any self-map of $T$.  As we make other connections between the differentiation operator $D$ and the specific composition operator $C_b$ in this section, the idea that the study of $D$ can lead to insight as to the behavior of $C_\varphi$ will be further supported.
\end{remark}

Once the boundedness of $D$ is established, it is natural to determine bounds on the norm or even an exact expression.  Determining the operator norm, or even establishing bounds, requires knowledge of the norms of the domain and codomain.  Lemma \ref{Lemma:DRepresentation} allows for the determination of bounds in terms of the norm of $C_b$.

\begin{corollary}\label{Corollary:OperatorNormBounds} 
Suppose $\X$ is a discrete function space on which $D$ is bounded. Then \[1-\|C_b\| \leq \|D\| \leq 1+\|C_b\|.\]
\end{corollary}

\begin{proof}
First observe that the identity operator $I$ is an isometry on $\X$, and thus is bounded and $\|I\|=1$.  Since $D$ is bounded on $X$, then so is $C_b$ by Theorem \ref{Theorem:BoundednessCphi}. To show the upper bound, observe 
\[\|D\| = \|I - C_b\| \leq \|I\| + \|C_b\| = 1 + \|C_b\|.\]
The lower bound follows since 
\[1 = \|I\| = \|D + C_b\| \leq \|D\| + \|C_b\|.\] The result then follows.
\end{proof}

\noindent Due to the representation $D=I-C_b$, we conjecture that the connection between $D$ and $C_b$ follows through the norm as follows.

\begin{conjecture}\label{Conjecture:OperatorNorm} 
Suppose $\X$ is a discrete function space on which $D$ is bounded. Then \[\|D\| = 1+\|C_b\|.\]
\end{conjecture}

\noindent We will see that Conjecture \ref{Conjecture:OperatorNorm} holds for the Lipschitz space (see Theorem \ref{Theorem:NormLipschitz}), the weighted Banach spaces (see Theorem \ref{Theorem:BoundednessLmuinf}), and the Hardy spaces (see Theorem \ref{Theorem:NormHardyp}).

We now consider the question of whether differentiation is an isometry when acting on a discrete function space.  First note that if a bounded linear operator $A$ is an isometry, then it is injective and $\|A\|=1$.  If Conjecture \ref{Conjecture:OperatorNorm} is true, then the differentiation operator acting on a discrete function space should not be an isometry. The condition that $\|D\|=1$ would only be true if $\|C_b\| = 0$.  Thus, we conjecture that $D$ is not an isometry on a discrete function space.

\begin{conjecture}\label{Conjecture:Isometry} 
The operator $D$ is not an isometry on any discrete function space $\X$.
\end{conjecture}

\noindent As with the operator norm, in order to prove or disprove Conjecture \ref{Conjecture:Isometry} we would need knowledge of $\|\cdot\|_\X$.  However, we can determine a class of discrete function spaces on which $D$ cannot be an isometry.

\begin{theorem}\label{Theorem:IsometriesWithConstants} 
If $\X$ is a discrete function space containing the constant functions then $D$ is not an isometry on $\X$.
\end{theorem}

\begin{proof}
All isometries must be injective, as a direct consequence of the definition.  The operator $D$ is not injective if the kernel of $D$ is non-trivial.  The only non-zero elements of $\ker(D)$ can be the constant functions by Lemma \ref{Lemma:DerivativeProperties}.  Thus, if $\X$ contains the constant functions, then $D$ is not injective, and thus not an isometry on $\X$.
\end{proof}

\subsection{Spectrum}
For a bounded linear operator acting on a Banach space, the spectrum offers vital information pertaining to invariant subspaces.  The interested reader is referred to \cite[Section 4.5]{MacCluer:2009}] or \cite[Section VII.6]{Conway:1990}.  We will collect information about the spectrum of the differentiation operator acting on a discrete function space of a tree.

\begin{remark}\label{Remark:Spectrum}
From \cite[Theorem 5.10]{MacCluer:2009}, the spectrum of such an operator $A$ is a non-empty, compact subset of $\C$ contained in the closed disk centered at 0 of radius $\|A\|$, that is \[\sigma(D) \subseteq \overline{D(0,\|A\|)} = \{\lambda : |\lambda| \leq \|A\|\}.\]
\end{remark}

The representation of $D$ as $I-C_b$ from Lemma \ref{Lemma:DRepresentation} has an immediate connection with the spectrum.  In fact, just as the operators $D$ and $C_b$ are connected, so are their spectra.

\begin{theorem}\label{Theorem:Spectrum} 
If $D$ is bounded on a discrete function space $\X$ then 
\[\sigma(D) = 1-\sigma(C_b) = \{1-\lambda : \lambda \in \sigma(C_b)\}.\]
\end{theorem}

\begin{proof}
First, suppose $\lambda$ in $\sigma(C_b)$.  We want to show $1-\lambda$ is an element of $\sigma(D)$.  Notice \[D-(1-\lambda)I = (I-C_b) - (1-\lambda) I = \lambda I - C_b.\] Since $\lambda \in \sigma(C_b), \lambda I - C_b$ is not invertible, making $1-\lambda$ an element of $\sigma(D)$.

Next, suppose $\lambda$ in $\sigma(D)$.  We want to show there exists $\mu$ in $\sigma(C_b)$ such that $\lambda = 1-\mu$.  This is equivalent to showing $1-\lambda \in \sigma(C_b)$.  This follows from the calculation \[(1-\lambda)I - C_b = (1-\lambda)I - (I-D) = D-\lambda I.\]  So $1-\lambda$ is an element of $\sigma(C_b)$.  Thus, we have shown $\sigma(D) = 1-\sigma(C_b)$, as desired. 
\end{proof}

\noindent Connecting the operator norm to the spectrum via Remark \ref{Remark:Spectrum}, we obtain the following bounding set for $\sigma(D)$. On the one hand, $\sigma(D)$ is a closed subset of the closed disk $\{\lambda: |\lambda| \leq \|D\|\}$.  Likewise, $\sigma(C_b)$ is a closed subset of $\{\lambda: |\lambda| \leq \|C_b\|\}$.  Then $\sigma(D) = 1-\sigma(C_b)$ is also a closed subset of $\{\lambda: |\lambda-1| \leq \|C_b\|\}$.

\begin{corollary}\label{Corollary:SpectrumBound} 
If $D$ is bounded on discrete function space $\X$ then $\sigma(D)$ is a closed subset of \[\{\lambda: |\lambda| \leq \|D\|\} \cap \{\lambda: |\lambda - 1| \leq \|C_b\|\}.\]
\end{corollary}

\noindent While Theorems \ref{Theorem:Spectrum} and \ref{Corollary:SpectrumBound} do not provide a means of determining the spectrum of $D$ on an arbitrary discrete function space, we can determine the eigenvalues of $D$ precisely.

\begin{theorem}\label{Theorem:Eigenvalues} 
If $D$ is bounded on a discrete function space $\X$ then 
\[\sigma_p(D) = \begin{cases}
\{0\} & \text{if $\X$ contains the constant functions}\\
\hfil \emptyset & \text{otherwise.}
\end{cases}\]
\end{theorem}

\begin{proof}
We will first show that the point spectrum of $D$ is a subset of $\{0\}$.  Assume, for purposes of contradiction, there exists $\lambda$ in  $\sigma_p(D)$ such that $\lambda \neq 0$.  We argue in the following two cases.

\begin{case}
Suppose $\lambda = 1$.  Then there exists non-zero function $f$ in $\X$ such that $Df = f$.  Let $w$ be a vertex in $T^*$, and observe \[f(w) - f(b(w)) = f'(w) = (Df)(w) = f(w).\] Thus $f(b(w)) = 0$ for all $w$ in $T^*$, which implies $f$ is identically 0, a contradiction.
\end{case}

\begin{case}
Suppose $\lambda \neq 1$.  Then there exists non-zero function $g$ in $\X$ such that $Dg = \lambda g$.  First observe \[\lambda g(o) = (Dg)(o) = g'(o) = 0.\] Thus $g(o) = 0$.  Next, let $v$ in $T^*$ with $|v|=1$.  Then \[\lambda g(v) = (Dg)(v) = g'(v) = g(v)-g(o) = g(v).\] It follows that $(1-\lambda)g(v) = 0$, and so $g(v) = 0$.  By induction on $|v|$, it follows that $g$ is identically 0, a contradiction.
\end{case}

To complete the proof, note that $0$ is an element of the point spectrum of $D$ if and only if there exists a non-zero function $f$ in $\X$ for which $Df \equiv 0$.  By Lemma \ref{Lemma:DerivativeProperties}, such a non-zero function exists if and only if $f$ is constant.
\end{proof}

\noindent Combining Theorems \ref{Theorem:Spectrum} and \ref{Theorem:Eigenvalues}, we arrive at the following corollary, which will be important in the study of the compact differentiation operators next.

\begin{corollary}\label{Corollary:Spectrum} 
If $D$ is bounded on a discrete function space $\X$ then
\begin{enumerate}
\item 0 is the only possible eigenvalue of $D$.
\item $\sigma(D)$ contains a non-zero element if and only if $\sigma(C_b)$ contains an element other than 1.
\end{enumerate}
\end{corollary}

\subsection{Compactness}
Finally, we determine conditions for which the differentiation operator on a discrete function space is compact.  Compact operators are of great interest in the study of operators, and the reader is referred to \cite[Chapter 4]{MacCluer:2009} or \cite[Chapters VI and VII]{Conway:1990}.

\begin{definition}[{\cite[Definition 4.5]{MacCluer:2009}}]
\label{Definition:Compact}
If $X$ and $Y$ are Banach spaces and $A:X \to Y$ is linear, we say $A$ is \textit{compact} if whenever $\{x_n\}$ is a bounded sequence in $X$, then $\{Ax_n\}$ has a convergent subsequence in $Y$.
\end{definition}

In the case $\X$ and $\Y$ are discrete function spaces on tree $T$, the following Lemma has been used extensively in characterizing compact multiplication and composition operators on the specific spaces considered in this paper.  This lemma is a modification of the result proved for Banach spaces of analytic functions in \cite[Lemma 2.10]{Tjani:1996}.

\begin{lemma}\label{Lemma:CompactnessLemma}
Let $X$ and $Y$ be Banach spaces of functions on tree $T$.  Suppose that
\begin{enumerate}
\item the point evaluation functionals are bounded,
\item the closed unit ball of $X$ is a compact subset of $X$ in the topology of uniform convergence on compact sets,
\item $A:X \to Y$ is bounded when $X$ and $Y$ are given the topology of uniform convergence on compact sets.
\end{enumerate} Then $A$ is a compact operator if and only if given a bounded sequence $\{f_n\}$ in $X$ such that $\{f_n\}$ converges to 0 pointwise, the sequence $\{Af_n\}$ converges to 0 in the norm of $Y$.
\end{lemma}

To determine if the differentiation operator is compact, one would need knowledge of the norm on $\X$ to use Definition \ref{Definition:Compact} or Lemma \ref{Lemma:CompactnessLemma}.  Even to show the operator is not compact using either of these results would require the construction of a sequence of functions in $\X$ with particular convergence, or norm, properties. It is for this reason we explore the connection of compactness and the spectrum through the Spectral Theorem.

\begin{theorem}[{\cite[Spectral Theorem for Compact Operators]{Conway:1990}}] 
If $X$ is an infinite-dimensional Banach space and $A$ is a compact operator acting on $X$, then one and only one of the following possibilities occurs:
\begin{enumerate}
\item $\sigma(A) = \{0\}.$
\item $\sigma(A) = \{0,\lambda_1,\dots,\lambda_n\}$, where for each $1 \leq k \leq n$, $\lambda_k$ is an eigenvalue of $A$ and $\ker(A-\lambda_kI)$ is finite-dimensional.
\item $\sigma(A) = \{0, \lambda_1, \lambda_2,\dots\}$, where for each $k \geq 1$, $\lambda_k$ is an eigenvalue of $A$, $\ker(A-\lambda_kI)$ is finite-dimensional, and $\lim_{k \to \infty} \lambda_k = 0$.
\end{enumerate}
\end{theorem}

\noindent We see the Spectral Theorem for Compact Operators provides information about eigenvalues of a compact operator.  An immediate consequence is that if $D$ is compact on an infinite-dimensional discrete function space $\X$ then any non-zero element of $\sigma(D)$ must be an eigenvalue.  Combining this with Corollary \ref{Corollary:Spectrum}, we see that if $\sigma(C_b)$ contains an element other than 1, then $D$ will have a non-zero element of the spectrum that is not an eigenvalue.  In this situation, $D$ cannot be compact.

\begin{corollary} 
Suppose $\X$ is an infinite-dimensional discrete function space.  If $\sigma(C_b)$ contains an element other than 1 then $D$ is not compact on $\X$.
\end{corollary}

While this corollary does not provide a characterization of compactness, it leads us to consider whether $C_b$ should have an element of its spectrum other than 1.  If $C_b$ is not invertible as an operator on $\X$, then $0$ would be an element of $\sigma(C_b)$. Then 1 would be an element of $\sigma(D)$ by Theorem \ref{Theorem:Spectrum}.  This would show $D$ is not compact on $\X$.

\begin{corollary} 
Suppose $\X$ is an infinite-dimensional discrete function space.  If $C_b$ is not invertible on $X$ then $D$ is not compact on $\X$.
\end{corollary}

As the composition operators on specific discrete function spaces are well studied, we conclude this section with conditions as to when $C_b$ is not invertible.  First, we show that $C_b$ is injective.  Thus, in order for $C_b$ to be non-invertible, it must be the case that $C_b$ is not surjective.

\begin{proposition} 
The operator $C_b$ is an injection on a discrete function space $\X$.
\end{proposition}

\begin{proof}
Assume, for purposes of contradiction, that $C_b$ is not injective as an operator on $X$.  Then there exists a non-zero function $f$ in $\ker(C_b)$.  Let $w$ in $T$ be such that $f(w) \neq 0$.  As $T$ contains no terminal vertices, there exists vertex $v$ in $T$ such that $b(v) = w$.  It then follows that \[f(w) = f(b(v)) = (C_b f)(v) = 0,\] a contradiction.  Thus, $C_b$ is an injection.
\end{proof}

We say that a function $f$ on a tree $T$ is \textit{constant on children} if for every $v$ in $T$, the function $f$ is constant on the set of children of $v$, that is for every $v$ in $T$, there exists constant $C(v)$ such that $f(w) = C(v)$ for all $w$ in $\ch(v)$.  The next result shows that if $C_b$ is surjective on a discrete function space, then every function must be constant on children. 

\begin{lemma}\label{Lemma:CbNotSurjective} Let $\X$ be a discrete function space on tree $T$. If $C_b:\X\to\X$ is surjective, then every function $f$ in $\X$ is constant on children.
\end{lemma}

\begin{proof}
Suppose $C_b$ is surjective on $\X$.  Let $f$ be an element of $\X$ and $v$ in $T$.  Suppose $w_1$ and $w_2$ in $\mathrm{ch}(v)$; that is $b(w_1) = v = b(w_2)$.  Since $C_b$ is surjective, there exists function $g$ in $\X$ such that $C_b g = f$.  Observe \[f(w_1) = (C_b g)(w_1) = g(b(w_1)) = g(v) = g(b(w_2)) = (C_b g)(w_1) = f(w_2).\] As $f$ and $v$ were arbitrary, every function of $\X$ is constant on children.
\end{proof}

Since it is reasonable to work toward proving that $D$ is not compact on a discrete function space $\X$, Lemma \ref{Lemma:CbNotSurjective} makes proving this equivalent to constructing a single function rather than a sequence of functions as when using Definition \ref{Definition:Compact} or Lemma \ref{Lemma:CompactnessLemma}.  

Next, we present a function on a class of tree $T$ that is not constant on children, namely the characteristic function on $w$ in $T$ defined by  \[\bigchi_w(v) = \begin{cases}
1 & \text{if $v=w$}\\
0 & \text{otherwise.}
\end{cases}\]  Suppose $T$ is a tree containing a vertex $v$ with at least 2 children, $u$ and $w$.  The function $\bigchi_w$ is not constant on the set $\ch(v)$, since $\bigchi_w(w) = 1$ but $\bigchi_w(u) = 0$.  Thus, on every such tree $T$ there exists a function $f$ that is not constant on children.

\begin{corollary}\label{Corollary:CharacteristicFunctionDNotCompact} Suppose $T$ is a tree containing a vertex $v$ with at least 2 children.  If $\X$ is an infinite-dimensional discrete function space on $T$ that contains $\bigchi_w$ for some $w$ in $\ch(v)$, then $D$ is not compact on $\X$.
\end{corollary}

We define a \textit{path tree} to be a tree for which every vertex has exactly one child.  Thus, Corolllary \ref{Corollary:CharacteristicFunctionDNotCompact} applies to an infinite discrete function space $\X(T)$ for which $T$ is not a path.  Recalling that the Hardy space $\HardypT$ is defined on a homogeneous tree of degree at least 2, any such tree on which the Hardy space is defined cannot be a path.  However, this is not necessarily the case for the Lipschitz space or the weighted Banach spaces.

\section{The Lipschitz Space}\label{Section:LipschitzSpace}
In this section, we study the differentiation operator acting on the Lipschitz space $\LipT$.  We utilize the results in Section \ref{Section:Operator}, determining that $D$ is neither an isometry nor compact on $\LipT$.  In addition, we determine the spectrum of $D$.  

Recall the definition of the Lipschitz space in Definition \ref{Definition:SpecificSpaces}.  It was proven in \cite{ColonnaEasley:2010} that the Lipschitz space is a functional Banach space under the norm \[\|f\|_\Lip = |f(o)| + \sup_{v \in T^*}\;|f'(v)|\] and point-evaluation bounds \[|f(v)| \leq |f(o)| + |v|\sup_{v \in T^*}\;|f'(v)|\] (see Theorem 2.1 and Lemma 3.4, respectively).  Also proven is the Lipschitz space is infinite-dimensional, as it contains a separable subspace (see Theorem 2.3).  Lastly, a direct calculation shows that the Lipschitz space contains the characteristic functions $\bigchi_w$ for all $w$ in $T$ with 
\[\|\bigchi_w\|_\Lip = \begin{cases}
2 & \text{if $w = o$}\\
1 & \text{otherwise}.
\end{cases}\]

As was observed in the previous section, we can approach determining if the differentiation operator $D$ is bounded on the Lipschitz space, either appealing to Theorem \ref{Theorem:BoundednessCGT} or Theorem \ref{Theorem:BoundednessCphi}.  As the norm on $\Lip$ involves the derivative already, it might be easier to determine when $C_b$ is bounded on $\Lip$.  

\begin{theorem}\label{Theorem:BoundednessLip} 
The operator $D$ is bounded on the Lipschitz space $\LipT$.
\end{theorem}

\begin{proof}
To show $D$ is bounded on $\LipT$ it suffices to show $C_b$ is bounded on $\LipT$ by Theorem \ref{Theorem:BoundednessCphi}.  By \cite[Theorem 3.2]{AllenColonnaEasley:2014}, $C_b$ is bounded on $\LipT$ if and only if \[\lambda_b = \sup_{v \in T^*} \metric(b(v),b(v^-)) < \infty.\]  Observe that 
\[\sup_{v \in T^*} \metric(b(v),b(v^-)) = \sup_{v \in T^*}\metric(b(v),b(b(v))) = \sup_{w \in T}\metric(w,b(w)).\]
If $w = o$, then $\metric(o,b(o)) = \metric(o,o) = 0$.  Otherwise, if $w$ in $T^*$, then $\metric(w,b(w)) = 1$.  So $\lambda_b = 1$. Thus $C_b$, and moreover $D$, is bounded on $\LipT$.
\end{proof}

\noindent Note that \cite[Theorem 3.2]{AllenColonnaEasley:2014} also establishes the norm of $C_b$, that is $\|C_b\| = \lambda_b = 1$. Thus, we can establish the following estimates on the norm of $D$ from Corollary \ref{Corollary:OperatorNormBounds}, namely \begin{equation}\label{Inequality:NormBounds}0 \leq \|D\| \leq 2.\end{equation} While the lower bound from Corollary \ref{Corollary:OperatorNormBounds} does not provide any information, the upper bound provides a reason to conjecture that $\|D\| = 2$, which would lend further evidence to support Conjecture \ref{Conjecture:OperatorNorm}. 

\begin{theorem}\label{Theorem:NormLipschitz}
The operator $D$ on the Lipschitz space $\LipT$ has norm 2.
\end{theorem}

\begin{proof}
From the bound established in \eqref{Inequality:NormBounds}, it suffices to show $\|D\| \geq 2$. This is done by constructing a function $f$ in $\LipT$ with $\|f\|_\Lip \leq 1$ and $\|Df\|_\Lip = 2$.  For a fixed vertex $w$ in $T^*$, $\|\bigchi_{w}\|_\Lip = 1$.  First observe 
\[\begin{aligned}
\bigchi''_{w}(v) &= \bigchi'_w(v) - \bigchi'_w(b(v))\\
&= \bigchi_w(v) - \bigchi_w(b(v)) - \left(\bigchi_w(b(v)) - \bigchi_w(b(b(v)))\right)\\
&= \bigchi_{w}(v) - 2\bigchi_{w}(b(v)) + \bigchi_{w}(b(b(v)))\\ 
&= \begin{cases}
\hfil 1 & \text{if $v = w$ or $b(b(v)) = w$}\\
\hfil -2 & \text{if $b(v) = w$}\\
\hfil 0 & \text{otherwise.}
\end{cases}
\end{aligned}\]
So
\[\|D\| \geq \|D\bigchi_w\|_\Lip = \|\bigchi'_w\|_\Lip
= \sup_{v \in T^*}\;|\bigchi''_w(v)|  = 2.
\] Thus $\|D\| = 2$, as desired.
\end{proof}

Since the Lipschitz space contains the constant functions, $D$ is not an isometry on $\LipT$ by Theorem \ref{Theorem:IsometriesWithConstants}.  Additionally, this follows from the fact that $\|D\| \neq 1$.

\begin{corollary} 
The operator $D$ is not an isometry on the Lipschitz space $\LipT$.
\end{corollary}

In the case of the Lipschitz space, we can completely determine the spectrum of $D$ through understanding the spectrum of $C_b$.  For the following theorem, we denote the closed disk in $\C$ of radius $r$ centered at $z_0$ by $\overline{D(z_0,r)}$, and will utilize the following results that we include for completeness.

\begin{theorem*}[{\cite[Theorem 5.1]{AllenColonnaEasley:2014}}] Let $\varphi$ be a non-constant function from a tree $T$ into itself. Then
the composition operator $C_\varphi$ is an isometry on $\LipT$ if and only if
\begin{enumerate}
\item $\varphi(o) = o$;
\item $\varphi(v)$ and $\varphi(w)$ are neighbors or $\varphi(v) = \varphi(w)$ whenever $v$ and $w$ are neighbors; and
\item $\varphi$ is onto.
\end{enumerate}
\end{theorem*}

\begin{theorem*}[{\cite[Theorem 6.2]{AllenColonnaEasley:2014}}] Let $C_\varphi$ be an isometry on $\LipT$.
\begin{enumerate}
\item If $\varphi$ is not an automorphism of $T$, then $\sigma(C_\varphi) = \overline{D}$.
\item If $\varphi$ is an automorphism of $T$, then the spectrum of $C_\varphi$ is contained in
the unit circle and the point spectrum is nonempty.
\end{enumerate}
\end{theorem*}

\begin{theorem} 
On the Lipschitz space $\LipT$ the spectrum of $D$ is $\overline{D(1,1)}$.
\end{theorem}

\begin{proof}
To show $\sigma(D) = \overline{D(1,1)}$, it suffices to show $\sigma(C_b) = \overline{\D}$ by Theorem \ref{Theorem:Spectrum}.  First, we prove that $C_b$ is an isometry on $\Lip$.  Observe that $b$ is a surjective self-map of $T$ fixing the root.  Let $v$ and $w$ be vertices in $T$ with $v$ and $w$ neighbors.  Without loss of generality, suppose $|v| = n$ for some $n \in \N_0$ and $w$ is a child of $v$.  If $v=o$, then $b(w) = o = b(v)$.  If $v \neq o$, $b(w)$ is a child of $b(v)$.  So $b(v)$ and $b(w)$ are neighbors. Thus by \cite[Theorem 5.1]{AllenColonnaEasley:2014} $C_b$ is an isometry on $\Lip$.  As $b$ is not injective, it is not an automorphism of $T$.  Therefore \cite[Theorem 6.2]{AllenColonnaEasley:2014} implies $\sigma(C_b) = \overline{\D}$.  Finally, by Theorem \ref{Theorem:Spectrum} we obtain $\sigma(D) = 1-\overline{\D} = \overline{D(1,1)}$, as desired.
\end{proof}

Finally, as a result of Corollary \ref{Corollary:CharacteristicFunctionDNotCompact} and the fact that the Lipschitz space contains the characteristic functions, the differentiation operator is not compact on $\LipT$ where $T$ is not a path tree.

\begin{corollary} 
If $T$ is not a path tree, then the operator $D$ is not compact on the Lipschitz space $\LipT$.
\end{corollary}

\section{Weighted Banach Spaces}\label{Section:WeightedBanachSpace}
In this section, we study the differentiation operator acting on the weighted Banach spaces $\Lmuinf$.  We utilize the results in Section \ref{Section:Operator}, demonstrating that $D$ is neither an isometry nor compact on any such weighted Banach space.

Recall the definition of the weighted Banach spaces in Definition \ref{Definition:SpecificSpaces}. It was proven in \cite{AllenCraig:2017} that each weighted Banach space is a functional Banach space under the norm \[\|f\|_\mu = \sup_{v \in T}\;\mu(v)|f(v)|\] and point-evaluation bounds \[|f(v)| \leq \frac{\|f\|_\mu}{\mu(v)}\] (see Theorem 2.1 and Proposition 2.2, respectively).  In addition, each weighted Banach space is infinite-dimensional, as they each contain a separable subspace (see \cite[Theorem 2.6]{AllenColonnaMartinez-AvendaPons:2019}).  Lastly, each weighted Banach space contains all the characteristic functions $\bigchi_w$ for $w$ in $T$ (see \cite[Lemma 2.7]{AllenPons:2018}). It is worth noting that when the weight $\mu \equiv 1$, the weighted Banach space $\Lmuinf$ is the space of bounded functions on $T$ with the supremum-norm, which has been denoted by $L^\infty$ beginning with the work in \cite{ColonnaEasley:2010}.  

To show the differentiation operator acting on $\Lmuinf$ is bounded, we can appeal to either Theorem \ref{Theorem:BoundednessCGT} by showing $D$ maps into $\Lmuinf$, or Theorem \ref{Theorem:BoundednessCphi} by showing $C_b$ is bounded on $\Lmuinf$.  As the composition operators on $\Lmuinf$ are well studied in \cite{AllenPons:2018}, we can arrive at a characterization directly by applying \cite[Theorem 3.1]{AllenPons:2018} to the backward shift map $b$.

\begin{corollary}\label{Corollary:BoundednessLmuinf} 
Suppose $\LmuinfT$ is a weighted Banach space.  Then $D$ is bounded on $\Lmuinf$ if and only if \[\sup_{v \in T} \frac{\mu(v)}{\mu(b(v))} < \infty.\]
\end{corollary}

\noindent Since $\|C_b\| = \sup_{v \in T}\frac{\mu(v)}{\mu(b(v))}$ by \cite[Theorem 3.1]{AllenPons:2018}, Corollary \ref{Corollary:OperatorNormBounds} results in the following bounds on the norm of $D$, namely \[1-\sup_{v \in T}\frac{\mu(v)}{\mu(b(v))} \leq \|D\| \leq 1 + \sup_{v \in T} \frac{\mu(v)}{\mu(b(v))}.\] Utilizing Theorem \ref{Theorem:BoundednessCGT} to prove boundedness also develops sharper bounds for the operator norm of $D$, resulting in an exact form. 

First note that Corollary \ref{Corollary:BoundednessLmuinf} is written as a direct application of \cite[Theorem 3.1]{AllenPons:2018} to the composition operator $C_b$ acting on $\Lmuinf$; the characterization of the boundedness of $D$ is written in terms of a supremum taken over all of $T$.  When considering the fact that $b(o) = o$, we see that $\frac{\mu(o)}{\mu(b(o))} = 1$.  Thus, \[\sup_{v \in T} \frac{\mu(v)}{\mu(b(v))} < \infty\] if and only if \[\sup_{v \in T^*}\frac{\mu(v)}{\mu(b(v))} < \infty.\] Although the boundedness of $D$ can be described in terms of $T$ or $T^*$, it is more natural to state $\|D\|$ in terms of $T^*$ since $f'(o) = 0$ for any function $f$ on a tree.

\begin{theorem}\label{Theorem:BoundednessLmuinf} 
Suppose $\LmuinfT$ is a weighted Banach space.  Then $D$ is bounded on $\Lmuinf$ if and only if \[\sup_{v \in T^*} \frac{\mu(v)}{\mu(b(v))} < \infty.\] Moreover, \[\|D\| = 1+\sup_{v \in T^*}\frac{\mu(v)}{\mu(b(v))}.\]
\end{theorem}

\begin{proof}
First, suppose $\sup_{v \in T^*}\frac{\mu(v)}{\mu(b(v))} < \infty$.  Let $f$ in $\Lmuinf$ with $\|f\|_\mu \leq 1$. Observe
\[\begin{aligned}
\|Df\|_\mu &= \sup_{v \in T}\;\mu(v)|f'(v)|\\
&= \sup_{v \in T^*}\;\mu(v)|f'(v)|\\
&= \sup_{v \in T^*}\;\mu(v)|f(v)-f(b(v))|\\
&\leq \sup_{v \in T^*}\;\mu(v)\left(|f(v)| + |f(b(v))|\right)\\
&\leq \sup_{v \in T^*}\;\mu(v)|f(v)| + \sup_{v \in T^*}\;\mu(v)|f(b(v))|\\
&\leq \|f\|_\mu + \sup_{v \in T^*} \frac{\mu(v)}{\mu(b(v))}\|f\|_\mu\\
&= \left(1+\sup_{v \in T^*}\frac{\mu(v)}{\mu(b(v))}\right)\|f\|_\mu\\
&\leq 1 + \sup_{v \in T^*}\frac{\mu(v)}{\mu(b(v))}.
\end{aligned}\]
So $Df$ is an element of $\Lmuinf$.  Thus Theorem \ref{Theorem:BoundednessCGT} implies $D$ is bounded on $\Lmuinf$.  In addition, we have established the upper bound $\|D\| \leq 1+\sup_{v \in T^*}\frac{\mu(v)}{\mu(b(v))}$.

Finally, suppose $D$ is bounded on $\Lmuinf$. Define the function $g(v) = \frac{(-1)^{|v|}}{\mu(v)}$.  By direct calculation, $\|g\|_\mu = 1$.  Observe
\[\begin{aligned}
1 + \sup_{v \in T^*} \frac{\mu(v)}{\mu(b(v))}
&= \sup_{v \in T^*}\;\left(1 + \frac{\mu(v)}{\mu(b(v))}\right)\\
&= \sup_{v \in T^*}\;\mu(v)\left(\frac{1}{\mu(v)} + \frac{1}{\mu(b(v))}\right)\\
&= \sup_{v \in T^*}\;\mu(v)\left|\frac{(-1)^{|v|}}{\mu(v)} - \frac{(-1)^{|v|-1}}{\mu(b(v))}\right|\\
&= \sup_{v \in T^*}\;\mu(v)\left|\frac{(-1)^{|v|}}{\mu(v)} - \frac{(-1)^{|b(v)|}}{\mu(b(v))}\right|\\
&= \sup_{v \in T^*}\;\mu(v)|g'(v)|\\
&= \sup_{v \in T}\;\mu(v)|g'(v)|\\
&= \|Dg\|_\mu\\
&\leq \|D\|.
\end{aligned}\]
Thus $\sup_{v \in T^*}\frac{\mu(v)}{\mu(b(v))}$ is finite and the lower bound of the operator norm is established, completing the proof.  
\end{proof}

Since the weight function $\mu$ is positive, $\frac{\mu(v)}{\mu(b(v))} > 0$ for all $v \in T^*$.  So $\|C_b\|$ is strictly positive, making $\|D\| > 1$.  Next, we show that for every value in $(1,\infty)$ there is an $\Lmuinf$ on which $D$ is bounded and $\|D\|$ attains the value.

\begin{theorem} 
For $M > 1$, there exists a weighted Banach space $\LmuinfT$ for which $D$ is bounded on $\Lmuinf$ and $\|D\| = M$.
\end{theorem}

\begin{proof}
Let $M > 1$ and define $\mu_{\text{\tiny$M$}}(v) = (M-1)^{|v|}$ for all $v \in T$.  Then a direct calculation shows $D$ is bounded on $L_{\hspace{-.2ex}\mu_{\text{\tiny$M$}}}^{\infty}$ since $\frac{\mu_{\text{\tiny$M$}}(v)}{\mu_{\text{\tiny$M$}}(b(v))} = M-1$ for all $v$ in $T^*$.  It then immediately follows that \[\|D\| = 1 + \sup_{v \in T^*} \frac{\mu_{\text{\tiny$M$}}(v)}{\mu_{\text{\tiny$M$}}(b(v))} = M,\] as desired.
\end{proof}

It is, in fact, not the case that $D$ is bounded on every weight Banach space, as the following example demonstrates.

\begin{example}
Define the weight $\mu$ on tree $T$ by \[\mu(v) = \begin{cases}|v| & \text{if $|v|$ is odd}\\\hfil 1 & \text{otherwise.}\end{cases}\] Let $(v_n)$ be a sequence in $T^*$ with $|v_n| = n$ for all $n$ in $\N$ and $b(v_n) = v_{n-1}$ for all $n \geq 2$.  Note such a sequence must exist since $T$ is an infinite tree with no terminal vertices.  By considering the subsequence $w_n = v_{2n+1}$, we see that 
\[\frac{\mu(w_n)}{\mu(b(w_n))} = \frac{\mu(v_{2n+1})}{\mu(v_{2n})} = \frac{|v_{2n+1}|}{1} = 2n+1 \to \infty\] as $n \to \infty$. It then immediately follows that $D$ is not bounded on $\Lmuinf$ by Theorem \ref{Theorem:BoundednessLmuinf}.
\end{example}

Having an exact expression of the operator norm for $D$ allows for the immediate determination that $D$ is not an isometry on $\Lmuinf$.  Theorem \ref{Theorem:IsometriesWithConstants} only applies to the weighted Banach spaces that contain the constant functions.  These are precisely the spaces induced by a bounded weight function $\mu$ \cite[Theorem 2.3]{AllenPons:2018}.  However, it does not apply to those whose weight function is unbounded on $T$, for example if $\mu(v) = |v|+1$.  However, since $\|D\| > 1$ on every weighted Banach space, it follows directly that $D$ is not an isometry. 

\begin{corollary} 
The operator $D$ is not an isometry on any weighted Banach space $\LmuinfT$.
\end{corollary}

Having information about the operator norms of $D$ and $C_b$ acting on weighted Banach spaces inform about the spectrum.  Combining results from the current section and Section \ref{Section:Operator}, we obtain the following.

\begin{corollary} 
Suppose $\LmuinfT$ is a weighted Banach space upon which $D$ is bounded.  Then
\begin{enumerate}
\item $\sigma_p(D) = \begin{cases}\{0\} & \text{if $\mu$ is bounded on $T$}\\\hfil\emptyset & \text{otherwise.}\end{cases}$
\item $\sigma(D)$ contains 1 and is a closed subset of $\left\{\lambda : |\lambda - 1| \leq \sup_{v \in T^*} \frac{\mu(v)}{\mu(b(v))}\right\}$.
\end{enumerate}
\end{corollary}

Finally, we see the differentiation operator is not compact on any weighted Banach space $\LmuinfT$, where $T$ is not a path, by Corollary \ref{Corollary:CharacteristicFunctionDNotCompact}, since $\Lmuinf$ contains the characteristic function $\bigchi_v$ for all $v$ in $T$.

\begin{corollary} 
If $T$ is not a path tree, then operator $D$ is not compact on weighted Banach space $\LmuinfT$.
\end{corollary}

\section{Hardy Spaces on Homogeneous Trees}\label{Section:HardySpaces}
In this section, we study the differentiation operator acting on the (discrete) Hardy spaces $\HardypT$ for $1 \leq p < \infty$, defined in \cite{MuthukumarPonnusamy:2017:I} on a $(q+1)$-homogeneous tree $T$ by Muthukumar and Ponnusamy.  We utilize the results in Section \ref{Section:Operator}, determining that $D$ is bounded on every Hardy space but not compact or an isometry on any.

In this section, we will always assume the tree $T$ is $(q+1)$-homogeneous without specifying, as the results apply for any $q$ in $\N$.  Also, we will only consider the case when $1 \leq p < \infty$.  The Hardy space $\T_\infty(T)$ is the space $L^\infty$, with the same norm.  So, the differentiation operator on $\T_\infty(T)$ is completely characterized in Section \ref{Section:WeightedBanachSpace} in the case that $T$ is a $(q+1)$-homogeneous tree and $\mu\equiv 1$.

Recall the definition of the Hardy spaces in Definition \ref{Definition:SpecificSpaces}. It was proven in \cite{MuthukumarPonnusamy:2017:I} that each such Hardy space is a functional Banach space under the norm \[\|f\|_p = \sup_{n \in \N_0}\;M_p(n,f)\] and point-evaluation bounds \[|f(v)| \leq \left((q+1)q^{|v|-1}\right)^{1/p}\|f\|_p\] (see Theorem 3.1 and Lemma 3.12, respectively).  In addition, each Hardy space is infinite-dimensional, as each space contains a separable subspace (see Theorem 3.10).  Lastly, each Hardy space contains all the characteristic functions $\bigchi_w$ for $w$ in $T$ (see proof of Theorem 3.10). 

To determine if the differentiation operator is bounded on any Hardy space, we will once again look to the composition operator $C_b$ acting on $\Hardyp$.  This operator is well studied in \cite{MuthukumarPonnusamy:2017:II} and \cite{MuthukumarPonnusamy:2020}.  For the particular operator $C_b$, there are several results that lead to determining boundedness.  We will utilize the one that also determines the operator norm, so that we can establish norm estimates for $D$ simultaneously.  

\begin{theorem}\label{Theorem:BoundednessHardyp}
For $1 \leq p < \infty$, the operator $D$ is bounded on the Hardy space $\HardypT$.
\end{theorem}

\begin{proof}
For $n$ in $\N$ and $w$ in $T$, define the set $P_{w,n}$ to be the number of vertices of length $n$ in $b^{-1}(w)$; that is $P_{w,n} = \left|b^{-1}(w) \cap \{|v|=n\}\right|$.  First note $b^{-1}(o) = \{o\} \cup \ch(o) = \overline{B(o,1)}$.  So $P_{o,0} = 1, P_{o,1} = q+1,$ and $P_{o,n} = 0$ for all $n \geq 2$.  Also, if $w \in T^*$ then $b^{-1}(w) = \ch(w)$.  So $P_{w,n} \neq 0$ only when $n = |w|+1$.  

For $m$ in $\N_0$, the constant $N_{m,n}$ is defined as $N_{m,n} = \max_{\{|w| = m\}} P_{w,n}$.  By the above discussion, $N_{m,n} \neq 0$ only when either $m=n=0$ or $n=m+1$.  So, we have $N_{0,0} = 1$ and \[N_{m,m+1} = \begin{cases}q+1 & \text{if $m=0$}\\q & \text{otherwise.}\end{cases}\] 

For each $n$ in $\N_0$, define the quantity $\alpha_n$ by \[\alpha_n = \frac{1}{c_n}\sum_{m=0}^\infty N_{m,n}c_m,\] where \[c_n = \begin{cases}1 & \text{if $n=0$}\\(q+1)q^{n-1} & \text{if $n \in \N$.}\end{cases}\]
By \cite[Theorem 4]{MuthukumarPonnusamy:2020} the composition operator $C_b$ is bounded on $\HardypT$ if and only if \[\alpha = \sup_{n \in \N_0}\;\alpha_n < \infty.\]
Observe $\alpha = 1$, and thus $C_b$ is bounded on $\Hardyp$, since for $n > 1$
\[\begin{aligned}
\alpha_0 &= \frac{1}{c_0}\sum_{m=0}^\infty N_{m,0}c_m = \frac{1}{c_0}\left(N_{0,0}c_0\right) = 1,\\
\alpha_1 &= \frac{1}{c_1}\sum_{m=0}^\infty N_{m,1}c_m = \frac{1}{c_1}\left(N_{0,1}c_0\right) = 1,\\
\alpha_n &= \frac{1}{c_n}\sum_{m=0}^\infty N_{m,n}c_m = \frac{1}{c_n}\left(N_{n-1,n}c_{n-1}\right) = \frac{qc_{n-1}}{c_n} = 1.
\end{aligned}\] 
Thus $D$ is bounded on $\HardypT$ by Theorem \ref{Theorem:BoundednessCphi}.
\end{proof}

\noindent Furhtermore, \cite[Theorem 4]{MuthukumarPonnusamy:2020} shows that $\|C_b\| = \alpha^{1/p} = 1$ for every space $\Hardyp$.  Thus, by Corollary \ref{Corollary:OperatorNormBounds}, we have \begin{equation}\label{Inequality:OperatorBoundsHardyp}0 \leq \|D\| \leq 2\end{equation} for the differentiation operator acting on $\Hardyp$.  If we can show that $\|D\| = 2$, then we again add evidence to support Conjecture \ref{Conjecture:OperatorNorm}.  

\begin{theorem}\label{Theorem:NormHardyp} For $1 \leq p < \infty$, the operator $D$ on the Hardy space $\HardypT$ has norm 2.
\end{theorem}

\begin{proof}
Since $\|D\| \leq 2$ by inequality \eqref{Inequality:OperatorBoundsHardyp}, to establish the lower bound $\|D\| \geq 2$ it suffices to construct a function $f$ in $\Hardyp$ with $\|f\|_p \leq 1$ for which $\|Df\|_p = 2$.  

Define the radially constant function $f$ on $T$ by 
\[f(v) = \begin{cases}
-1 & \text{if $v=o$}\\
1 & \text{if $|v|=1$}\\
0 & \text{otherwise.}
\end{cases}\] 
We see that $\|f\|_p = 1$ from the following calculation, with $m \geq 2$,
\[\begin{aligned}
M_p(0,f) &= |f(o)| = 1\\
M^p_p(1,f) &= \frac{1}{q+1}\sum_{|v|=1}|f(v)|^p = \frac{1}{q+1}(q+1)(1)^p = 1\\
M^p_p(m,f) &= \frac{1}{(q+1)q^{m-1}}\sum_{|v|=m}|f(v)|^p = 0.
\end{aligned}\] 

\noindent Note that \[(Df)(v) = \begin{cases}2 & \text{if $|v|=1$}\\
-1 & \text{if $|v|=2$}\\ 0 & \text{otherwise.}\end{cases}\]  From the following calculations, we see $\|Df\|_p = 2$.  Observe, with $m \geq 3$,
\[\begin{aligned}
M_p(0,f') &= |f'(o)| = 0\\
M^p_p(1,f') &= \frac{1}{q+1}\sum_{|v|=1}|f'(v)|^p = \frac{1}{q+1}(q+1)2^p = 2^p\\
M^p_p(2,f') &= \frac{1}{(q+1)q}\sum_{|v|=2}|f'(v)|^p = \frac{1}{(q+1)q}(q+1)q(1^p) = 1\\
M^p_p(m,f') &= \frac{1}{(q+1)q^{m-1}}\sum_{|v|=m}|f'(v)|^p = 0.
\end{aligned}\]
Thus $\|D\| = 2$, as desired.
\end{proof}

From a direct calculation, we see that if $f$ is a constant function, then $\|f\|_p = |f(o)|$ since for each $n$ in $\N$ we have \[M_p^p(n,f) = \frac{1}{(q+1)q^{n-1}}\sum_{|v|=n} |f(v)|^p = \frac{1}{(q+1)q^{n-1}}\left((q+1)q^{n-1}|f(o)|^p\right) = |f(o)|^p.\]  Thus, by Theorem \ref{Theorem:IsometriesWithConstants}, the differentiation operator acting on any Hardy space is not an isometry.

\begin{corollary} 
For $1 \leq p < \infty$, the operator $D$ is not an isometry on the Hardy space $\HardypT$.
\end{corollary}

Since $\HardypT$ contains the constant functions for each $1 \leq p < \infty$ and $q$ in $\N$, the point spectrum of $D$ is precisely $\{0\}$.  As $\HardypT$ also contains the characteristic functions, $C_b$ is not surjective on $\Hardyp$, and thus $1$ is an element of $\sigma(D)$.  Applying the results of Corollaries \ref{Corollary:SpectrumBound} and \ref{Corollary:Spectrum}, we obtain the following result.

\begin{corollary}\label{Corollary:SpectrumHardyp} 
Suppose $\HardypT$ is a Hardy space, for $1 \leq p < \infty$ .  Then
\begin{enumerate}
\item $\sigma_p = \{0\}$
\item $\sigma(D)$ contains 1 and is a closed subset of $\{\lambda:|\lambda-1|\leq 1\}.$
\end{enumerate}
\end{corollary}

Finally, since each Hardy space contains the characteristic functions, the differentiation operator acting on $\Hardyp$ cannot be compact by Corollary \ref{Corollary:CharacteristicFunctionDNotCompact}.

\begin{corollary} 
For $1 \leq p < \infty$, the operator $D$ is not compact on $\HardypT$.
\end{corollary}

\section{Open Questions}\label{Section:OpenQuestions}
We conclude with open questions which were inspired while developing this manuscript.  Our hope is that these questions will initiate further research in this area of operator theory.

\begin{enumerate}
\item[1.] Is there a discrete function space on which $D$ is bounded, but its norm is not $1 + \|C_b\|$?
\item[2.] Is there a non-trivial discrete function space on which $D$ acts isometrically?
\item[3.] Is there a non-trivial discrete function space on which $C_b$ is surjective?
\item[4.] Is there a non-trivial discrete function space on which $D$ is compact?
\item[5.] What is the spectrum of $D$ or $C_b$ acting on either $\Lmuinf$ or $\HardypT$?
\item[6.] Is $D$ compact on $\LipT$ or $\LmuinfT$ for path tree $T$?
\end{enumerate}

\section*{Acknowledgements}
The work of the second author was conducted while an undergraduate student at the University of Wisconsin-La Crosse and funded by the Department of Mathematics \& Statistics Undergraduate Research Endowment.

\bibliographystyle{amsplain}
\bibliography{references.bib}
\end{document}